\documentclass[11pt]{article}

\usepackage{amsmath,amssymb}
\usepackage{amsthm}
\usepackage{setspace}

\begin{document}

\newtheorem{thm}{Theorem}
\newtheorem{claim}[thm]{Claim}
\newtheorem{cons}[thm]{Construction}
\newtheorem{conj}[thm]{Conjecture}
\newtheorem{cor}[thm]{Corollary}
\newtheorem{exa}[thm]{Example}
\newtheorem{lem}[thm]{Lemma}
\newtheorem{obs}[thm]{Observation}
\newtheorem{prop}[thm]{Proposition}

\def\cA{{\cal A}}
\def\cB{{\cal B}}
\def\cC{{\cal C}}
\def\cE{{\cal E}}
\def\cF{{\cal F}}
\def\cG{{\cal G}}
\def\cH{{\cal H}}
\def\ck{{\cal K}}
\def\cI{{\cal I}}
\def\cJ{{\cal J}}
\def\cL{{\cal L}}
\def\cM{{\cal M}}
\def\cN{{\cal N}}
\def\cP{{\cal P}}
\def\cQ{{\cal Q}}
\def\cS{{\cal S}}

\voffset=-0.5in

\pagestyle{myheadings}
\markright{{\small \sc F\"uredi, Kostochka, Kumbhat:}
  {\it\small Abundant packings and choosability with separation }}

\title{\huge Minimal abundant packings\\
  and  choosability with separation  
} 

\author{{\bf Zolt\'an F\"uredi}\footnote{
Research supported in part by the National Research Development and Innovation Office, NKFIH,
KKP 133819 and OTKA 132696.
The support of the HUN-REN Research Network is appreciated.}
\\
{\small\em
R\'enyi Mathematical Institute, Budapest, Hungary}\\
{\small\texttt furedi.zoltan@renyi.hu, z-furedi@illinois.edu}
\\ \\
{\bf Alexandr Kostochka}\footnote{
Research of this author
was supported in part by  NSF  Grant DMS-2153507 and by NSF RTG Grant DMS-1937241.}\\
{\small\em
University of Illinois at Urbana--Champaign, Urbana, IL 61801}\\
{\small\texttt kostochk@illinois.edu}\\
\\
{\bf Mohit Kumbhat}\\
{\small\em University of Nevada at Reno,  Reno, NV 89557}\\
{\small\texttt  mkumbhat@unr.edu }\\
}

\date{${}$}

\renewcommand{\thefootnote}{\empty}

\maketitle

\begin{abstract}
A $(v,k,t)$ packing of size $b$ is a system of $b$ subsets (blocks) of
 a $v$-element underlying set such that each block has $k$ elements
 and every $t$-set is contained in at most one block.
$P(v,k,t)$ stands for the maximum possible $b$.
A packing is called {\it abundant} if $b> v$.
We give new estimates for $P(v,k,t)$ around the critical range,
 slightly improving the Johnson bound and asymptotically determine the minimum $v=v_0(k,t)$ when {\it abundant} packings exist.

For a graph $G$ and a positive integer $c$,
 let $\chi_\ell(G,c)$ be the minimum value of $k$ such that one can
 properly color the vertices of $G$ from any assignment of lists $L(v)$ such that
 $|L(v)|=k$ for all $v\in V(G)$ and
  $|L(u)\cap L(v)|\leq c$ for all $uv\in E(G)$.
 Kratochv\'{\i}l, Tuza and Voigt in 1998 asked to determine $\lim_{n\rightarrow \infty}
 \chi_\ell(K_n,c)/\sqrt{cn}$ (if exists). Using our bound on $v_0(k,t)$, we prove that the limit exists and equals  $1$.
Given $c$, we find the exact value of $\chi_\ell(K_n,c)$ for infinitely many  $n$.
\end{abstract}

\section{Preliminaries on hypergraphs}\label{s1}

A {\em hypergraph} $\cH=(V,\cE)$ consists of a set of vertices $V=V(\cH)$ and a collection $\cE$ of subsets of $V$ called edges or blocks, i.e., multiple copies of edges are allowed.
Often we take $V(\cH)=[v]$, where  $[v]:=\{ 1, 2, 3,\dots, v\}$.
The {\em degree} of a vertex $x\in V$, denoted by $d_\cH(x)$ or just by $d_x$, is the number of edges containing the vertex $x$.

A set of distinct vertices $\{ x_1,, \dots, x_m \}$ is called a {\em system of distinct representatives} (SDR, for short) of the (multi)family $\cE=\{ E_1, \dots, E_m\}$ if $x_i\in E_i$
for each $i\in [m]$. By classic Hall's Theorem~\cite{Hall}, $\cE$ has an SDR if and only if it satisfies Hall's condition
\begin{equation}\label{eq1}
  |\bigcup_{E\in \cE'} E| \geq |\cE'| \enskip \text{ for all } \cE'\subseteq \cE.
  \end{equation}

A hypergraph is $k$-{\em uniform} if all of its edges has $k$ elements.
It is a $t$-{\em packing} if $|E\cap E'|<t$ for any two distinct edges $E,E'\in \cE$.
The following theorem is usually attri\-buted to Johnson~\cite{Johnson},
who used it to get upper bounds for error-correcting codes.
It was rediscovered several times, e.g., Bassalygo~\cite{Ba65}, Corr\'adi~\cite{Co}. 
Let $\cE:=\{ E_1, \dots, E_m \}$ be a family of $k$-sets such that $|E_i\cap E_j|<t$
  for all $1 \leq i <j\leq m$.
Then
\begin{equation}\label{eq2}
  v:=|\bigcup_{i=1}^m E_i|\ge \frac{mk^2}{(m-1)(t-1)+k}.
  \end{equation}

A $(v,k,t)$ packing of size $b$ is a system of $b$ subsets (blocks) of
 a $v$-element underlying set such that each block has $k$-elements
 and every $t$-set is contained in at most one block.
$P(v,k,t)$ stands for the maximum possible $b$.
A packing is called {\it abundant} if $b> v$.
For example, the finite affine plane $AG(2,q)$ of order $q$ is a (perfect) $(q^2, q,2)$ packing with $q^2+q$ blocks, so it is abundant.

Let $v_0(k,t)$ stand for the minimum $v$ that $P(v,k,t)>v$.
For example, we have $v_0(q,2)\leq q^2$ if an $AG(2,q)$ exists.
Applying~\eqref{eq2} to $v+1$ blocks of an abundant packing one gets $v\geq (v+1)k^2 /(v(t-1)+k)$. Rearranging we get $v\left( v(t-1)-k^2+k\right)\geq k^2$ and thus
\begin{equation}\label{eq3}
v_0(k,t)> (k^2-k)/(t-1).
    \end{equation}

\section{Main result and an application}\label{s2}

Our main aim here is to show that~\eqref{eq3} gives the true order of magnitude of $v_0$.

\begin{thm}\label{th2.1}
Let $t\geq 2$ and suppose that $k\to \infty$. Then $v_0(k,t)= (1+o(1))\dfrac{k^2}{t-1}$.
\end{thm}

Leaving out an arbitrary element from each block of a $(v,k,t)$ packing one obtains a $(v,k-1,t)$ packing of the same size (for $k>t\geq 2$). In this way one can see that the sequence $\{ v_0(k,t): k=t, t+1, t+2, \dots \} $ is strictly increasing. So 
 Theorem~\ref{th2.1} follows from~\eqref{eq3} and an explicit construction of an infinite series of abundant packings for a dense sequence of $k$'s giving us an asymptotically matching upper bound.
From now on, it will be more convenient to use $c$ for $t-1$. The following construction is presented in Section~\ref{sec4}.

\begin{cons}\label{cons21}
Let $c\ge 1$ and suppose that $q$ is a prime power, $c<q-1$ and $c$ divides $q-1$.
Then
$$P\left(\frac{q^2-1}{c}+1, q, c+1\right) \geq  \frac{q^2-1}{c} + \left\lfloor \frac{q+1}{c} \right\rfloor.
$$
\end{cons}

We obtain for these values that
\begin{equation}\label{eq4}
 v_0(q,c+1)\leq \frac{q^2-1}{c}+1.
 \end{equation}
It is known (see~\cite{HI}) that for every sufficiently large real $k$
there exists a prime $q\in [k, k+k^{0.6}]$
 such that $c$ divides $q-1$.
Then the monotonicity of $v_0$ and~\eqref{eq4} yield
$$  v_0(k,c+1) \leq v_0(q,c+1)\leq \frac{q^2-1}{c}+1 < \frac{k^2}{c} + O(k^{1.6}).
  $$
This together with~\eqref{eq3}, completes the proof of Theorem~\ref{th2.1}.

Theorem~\ref{th2.1} allows to answer a question on list colorings of graphs.
Recall that a {\em list  $L$ for}  a graph $G$ is an assignment to every $v\in
V(G)$, a set $L(v)$
 of colors that may be used for coloring  $v$. Graph  $G$ is $L$-{\em colorable,}
if there exists a proper coloring $f$ of the vertices of $G$
from $L$, i.e., if $f(v)\in L(v)$ for all $v\in V(G)$ and
$f(u)\neq f(v)$ for all $uv\in E$.
The {\em list chromatic number} of $G$, $\chi_\ell(G)$, is the least $k$ such that $G$ is $L$-colorable,
whenever $|L(v)|=k$ for all $v\in V(G)$.

A list $L$ for a graph $G$ is a {\em $(k,c)$-list} if
$|L(v)|=k$ for all $v\in V(G)$ and $|L(u)\cap L(v)|\leq c$ for all $uv\in E(G)$.
Kratochv\'{\i}l, Tuza and Voigt~\cite{KTV} introduced $\chi_\ell(G,c)$, the least $k$ such that
$G$ is $L$-colorable from each $(k,c)$-list $L$.
They showed that  $\sqrt{cn/2}\le \chi_\ell(K_n,c) \le \sqrt{2ecn}$,
where $K_n$ is the complete graph on $n$ vertices.
They asked whether the limit $\lim_{n\rightarrow \infty}
\chi_\ell(K_n,c)/\sqrt{cn}$ exists.
In Section~\ref{sec5} we use  Theorem~\ref{th2.1} to prove that the limit exists and is  $1$.
We also find the exact value of $\chi_\ell(K_n,c)$ for infinitely many values of $n$.

\section{Explicit gaps}\label{sec3}

There are many results concerning packings when equality holds in~\eqref{eq2}. 
These packings have $1+ (k^2-k)/(t-1)$ blocks and are called {\em symmetric} $(t-1)$-designs, see~\cite{Hall}.
By improving~\eqref{eq3} we establish large explicit gaps between $v_0(k,t)$ and $v_0(k+1,t)$ for many cases. These gaps will be used in our second topic concerning $(k,c)$-list colorings of graphs (see Section~\ref{sec5}).

\begin{claim}\label{cl21}
Let $q>c\geq 1$. Then
 $v_0(q+1,c+1) \geq \dfrac{1}{c}\left( q^2+q+ \dfrac{2(q-c+1)}{c+1}  \right)+1$.
\end{claim}
Note that using~\eqref{eq2}, i.e., the inequality
$v\geq (v+1)(q+1)^2/(vc +q+1)$, leads to the bound $v_0 > (q^2+q)/c + O(1)$.
So we have a slight improvement on the Johnson bound in this critical range.

One can summarize~\eqref{eq4} and Claim~\ref{cl21} in one formula:
\begin{equation}\label{eq5}
v_0(q,c+1)\leq \frac{q^2-1}{c}+1 < \frac{1}{c}\left( q^2+q+ \frac{2(q-c+1)}{c+1} \right)+1 \leq v_0(q+1,c+1)
    \end{equation}
whenever $q$ is a prime power, $1\leq c<q-1$, and $c$ divides $q-1$.

\begin{lem}\label{vertex counting lemma}
Let $c\ge 1$ and suppose that $\cE$ is a $q$-uniform hypergraph on vertex set $Y$ such that
 $|\cE|=q+2$ and $|E\cap E'|\le c-1$ for any two edges.
 Then $|Y|\ge
\frac{1}{c} (q^2+q+\frac{2(q-c+1)}{c+1})$.
\end{lem}

\begin{proof}
Let $d_y$ be the degree of the vertex $y$. We have
\begin{equation}\label{eq6}
\sum_{y\in Y}{d_y\choose 2} = \sum_{E,E'\in \cE} |E\cap E'|\le (c-1){q+2 \choose 2},
\end{equation}
\begin{equation}\label{eq7}
\sum_{y\in Y}{d_y} = \sum_{E\in \cE} |E|= (q+2)q.
\end{equation}

Multiply (\ref{eq6}) by $-2$,  (\ref{eq7}) by $2c$, add them up and rearrange. We get
$$
  c(c+1)|Y| + \sum_{y} -(d_y-c)(d_y-c-1) \geq -(c-1)(q+2)(q+1)+ 2c(q+2)q.
  $$
Rearranging we get the desired lower bound for $|Y|$.
\end{proof}

\begin{proof}[Proof of Claim~\ref{cl21}]
Let $\cP$ be an abundant $(v,q+1,c+1)$ packing on the vertex set $V$. 
Since 
$$
   \sum_{x\in V} d_x =\sum_{P\in \cP} |P| = |\cP|(q+1)> v(q+1), 
   $$  
there exists an $x\in V$ with $d_x> q+1$.
So one can find $q+2$ edges of $\cP$ of the form $\{x\} \cup E_i$ where the family
 $\{ E_1, \dots, E_{q+2}\}$ is a $(v-1, q, c)$ packing on the vertex set $Y=V\setminus \{ x\}$.
One can now apply Lemma~\ref{vertex counting lemma} to complete the proof.
\end{proof}

\section{Construction of a packing}\label{sec4}

In this section we present Construction~\ref{cons21},
a $\left(\frac{q^2-1}{c}+1, q, c+1\right)$ packing $\cP$ of size
$\frac{q^2-1}{c} + \left\lfloor \frac{q+1}{c} \right\rfloor$
whenever $c\ge 1$, $q$ is a prime power, $c<q-1$, and $c$ divides $q-1$.

 Let $\textbf{F}$ be the $q$-element finite field and
let $g$ be an element of order $c$ in the multiplicative group $\textbf{F}\backslash
\{0\}$. Set $H=\{1,g,g^2,...,g^{c-1}\}$, it is a $c$-element subgroup of
$\textbf{F}\backslash \{0\}$. For $(a,b),(a',b')\in
(\textbf{F}\times\textbf{F})$ we say that $(a,b)\sim (a',b')$ if
there exists an $h\in H$ such that $(a',b')=(ha,hb)$.
This is an equivalence relation with $\{ (0,0)\}$ being a 1-element class. Each other
equivalence class is a collection of $c$ elements in ($\textbf{F}\times\textbf{F})\backslash \{(0,0)\}$. So there are $1+ (q^2-1)/c$ equivalence classes.
The equivalence class containing $(a,b)$ is denoted by $\langle a,b \rangle$.
These equivalence classes form the vertex set $V$ of the packing $\cP$.

For $(a,b)\neq (0,0)$, define the set $L\langle a,b \rangle = \{\langle x,y \rangle : ax+by\in
H\}$. Since $H$ is a group, $ax+by\in H$ implies $(h'a)x+(h'b)y\in H$, for all $h'
\in H$. Hence $L\langle a,b \rangle$ is a well-defined subset of $V$.
The next statement is a consequence of basic linear algebra.

\begin{claim}[Furedi~\cite{F}]~\label{Furedi}
Let $(V,\mathcal{L})$ be the hypergraph with vertex set $V=\{\langle a,b \rangle:
a,b\in \textbf{F}\}$ and edge set $\cL=\{L\langle a,b \rangle:
a,b\in \textbf{F}, (a,b)\neq(0,0)\}$. Then\\
$(i)$  $\cL$ is a $q$-uniform hypergraph, $|L\langle a,b \rangle|=q$,\\
$(ii)$ $V$ has $1+(q^2-1)/c$ vertices, \\
$(iii)$ $\cL$ has $(q^2-1)/c$ edges.\\
$(iv)$ Suppose that $(a,b)\nsim (a',b')$. Then $|L\langle a,b \rangle \cap L\langle
a',b' \rangle|= c$ whenever {\rm det}$\left( \begin{array}{cc}
a & b \\
a' & b' \\ \end{array} \right)\neq 0$ and
 $|L\langle a,b \rangle \cap L\langle a',b' \rangle|=0$ whenever this determinant is $0$.
\end{claim}

Define the sets $V_m:=\{\langle x,y \rangle: y=mx, (x,y)\neq (0,0)\}$ for $m\in \textbf{F}$ and let
  $V_\infty:= \{\langle x,y \rangle: x=0, (x,y)\neq (0,0)\}$.
Then $|V_\alpha|=(q-1)/c$ and these sets form a partition of $V\setminus \langle 0,0\rangle$.
Moreover, $|V_\alpha\cap L\langle a,b \rangle|\leq 1$ for each $\langle
a,b \rangle \in V$.

Select $\lfloor (q+1)/c \rfloor$ disjoint $c$-sets $C_1, C_2, \dots$ from $\textbf{F} \cup\{ \infty\}$
 and define $L(i):= \cup \{ V_\alpha: \alpha \in C_i\}\cup \{ \langle 0,0 \rangle\} $.
Then these are $q$-element sets pairwise meeting in  $\langle 0,0 \rangle$.
Moreover, $|L(i)\cap L\langle a,b \rangle|\leq c$. Finally,
  $\cP:=\cL \cup \{ L(1), L(2), \dots \}$ is a packing we were looking for.

\section{List colorings}\label{sec5}

In this section we answer the question of Kratochv\'{\i}l, Tuza and Voigt~\cite{KTV} on colorings of complete graphs from $(k,c)$-lists.

\begin{thm}\label{main theorem}
Let $c\ge 1$, then\\
$(i)$ $\lim_{n\rightarrow \infty} \chi_\ell(K_n,c)/\sqrt{cn} =1$. \\
$(ii)$ If $q$ is a prime power, $c<q-1$ and $c$ divides $q-1$, then
$\chi_\ell(K_n,c)=q+1$ for all
$$n\in
\left[ \frac{q^2-1}{c}+2, \frac{1}{c}\left( q^2+q+ \frac{2(q-c+1)}{c+1} \right)+1\right].$$
\end{thm}

\begin{proof}
The complete graph $K_n$ is $L$-colorable if and only if the set of lists $\{ L(v): v\in [n]\}$
 satisfy Hall's condition~\eqref{eq1}. (This observation is due to Vizing~\cite{V}.)
A $(k,c)$-list corresponds to a $(c+1)$-packing of $k$-sets.
So $\chi_\ell(K_n,c)>k$ if and only if there is an abundant $(v,k,c+1)$ packing with $v\leq n$. Hence
\begin{equation}\label{eq8}
   \chi_\ell(K_n,c)= q+1 \enskip \Longleftrightarrow \enskip v_0(q, c+1)< n \leq v_0(q+1,c+1).
  \end{equation}

To make~\eqref{eq8} more clear, let us explain. 
If $v_0(q, c+1)< n$, then there exists a $(v,q,c+1)$ packing $\cP$
 of size $v+1\leq n$.
Assign the members of $\cP$ to the first $v+1$ vertices of $K_n$ and assign completely disjoint $q$-sets to the rest of the vertices. This assignment does not satisfy Hall's condition, so we obtain $\chi_\ell(K_n,c) >q$.
On the other hand, if $n \leq v_0(q+1,c+1)$ then any  $(q+1, c)$-list assignment $L$ of $K_n$ is a
 $(c+1)$-packing of $(q+1)$-sets of size at most $v_0(q+1,c+1)$.  So neither $\{ L(v): v\in [n]\}$ is abundant, nor any part of it is abundant. Therefore, it satisfies Hall's condition and thus implying $K_n$ is $L$-colorable.

Finally, Theorem~\ref{main theorem} is now a corollary of  Theorem~\ref{th2.1}, \eqref{eq5}, and~\eqref{eq8}.
\end{proof}

For a fixed $c\ge 1$, one might be interested in knowing what is the
maximum value of $\chi_\ell(G,c)$ over all $n$-vertex graphs $G$. Note that if $H$ is
an induced subgraph of $G$, then $\chi_\ell(H,c)\le \chi_\ell(G,c)$, but this may not
hold true for non-induced subgraphs. We have the following conjecture.

\begin{conj}
If $c,n\ge 1$ and $G$ is an $n$-vertex graph, then $\chi_\ell(G,c)\le \chi_\ell(K_n,c)$.
\end{conj}

Work~\cite{KTV} generated lots of further research, especially concerning planar graphs, 
e.g.,~\cite{YWW}. 
For further recent results concerning separated list colorings see~\cite{DEKO,EKT}.

\end{document}